\documentclass[11pt, twoside]{article}
\usepackage{latexsym}
\usepackage{amsmath}
\usepackage{amssymb}
\usepackage[all]{xy}
\usepackage{amsfonts}
\usepackage{verbatim}
\usepackage{amsthm}
\usepackage{mathrsfs}
\usepackage{epsfig}
\usepackage{xy}
\usepackage{array}
\usepackage{stmaryrd}
\usepackage{graphicx,color}
\usepackage{xcolor}
\usepackage[colorlinks=true,linkcolor=blue,citecolor=blue]{hyperref}
\usepackage{tikz}
\usetikzlibrary{arrows,calc}
\usepackage{etex}
\usepackage{mathdots}
\usepackage{float}
\usepackage{graphics}
\usepackage{pdflscape}
\usepackage{bm}
\usepackage{anysize,hyperref}
\input xypic
\xyoption{all}
\newcommand{\mr}{\hbox{\boldmath$\cdot$}}
\usepackage[perpage,symbol]{footmisc}
\usepackage{setspace}
\usepackage{enumitem}

\def\A{\mathscr{A}}

\def\C{\mathscr{C}}
\def\E{\mathbb{E}}

\def\s{\mathfrak{s}}
\def\id{\mathrm{id}}
\def\op{^\mathrm{op}}
\def\Ab{\mathit{Ab}}
\def\del{\delta}
\def\dr{\ar@{->}[r]}

\def\Im{\mbox{Im}}
\def\Ker{\mbox{Ker}}

\newcommand{\CC}{{\bf{C}}^{n+2}_{\C}}

\newcommand{\ov}{\overset}
\newcommand{\lra}{\longrightarrow}
\newcommand{\co}{\colon}
\newcommand{\uas}{^{\ast}}            
\newcommand{\sas}{_{\ast}}
\newcommand{\Xd}{\langle X^{\mr},\del\rangle}  
\newcommand{\Yr}{\langle Y^{\mr},\rho\rangle}  

\SelectTips{cm}{10}

\begin{document}
\title{\Large{\bf Idempotent completion of certain $\bm{n}$-exangulated categories\footnotetext{\hspace{-1em}Jian He was supported by the National Natural Science Foundation of China (Grant No. 12171230). Panyue Zhou was supported by the National Natural Science Foundation of China (Grant No. 11901190).} }}
\medskip
\author{Jian He, Jing He and Panyue Zhou}

\date{}

\maketitle
\def\blue{\color{blue}}
\def\red{\color{red}}

\newtheorem{theorem}{Theorem}[section]
\newtheorem{lemma}[theorem]{Lemma}
\newtheorem{corollary}[theorem]{Corollary}
\newtheorem{proposition}[theorem]{Proposition}
\newtheorem{conjecture}{Conjecture}
\theoremstyle{definition}
\newtheorem{definition}[theorem]{Definition}
\newtheorem{question}[theorem]{Question}
\newtheorem{remark}[theorem]{Remark}
\newtheorem{remark*}[]{Remark}
\newtheorem{example}[theorem]{Example}
\newtheorem{example*}[]{Example}
\newtheorem{condition}[theorem]{Condition}
\newtheorem{condition*}[]{Condition}
\newtheorem{construction}[theorem]{Construction}
\newtheorem{construction*}[]{Construction}

\newtheorem{assumption}[theorem]{Assumption}
\newtheorem{assumption*}[]{Assumption}

\baselineskip=17pt
\parindent=0.5cm
\vspace{-6mm}

\begin{abstract}
\baselineskip=16pt
It was shown recently that
an $n$-extension closed subcategory $\A$ of a Krull-Schmidt $(n+2)$-angulated category
has a natural structure of an $n$-exangulated category. In this article, we prove that its idempotent completion $\widetilde{\A}$ admits an $n$-exangulated structure.
It is not only a generalization of the main result of Lin, but also gives an $n$-exangulated category which is neither $n$-exact nor $(n+2)$-angulated in general.\\[0.2cm]
\end{abstract}

\pagestyle{myheadings}
\markboth{\rightline {\scriptsize   Jian He, Jing He and Panyue Zhou}}
         {\leftline{\scriptsize Idempotent completion of certain $n$-exangulated categories}}

\section{Introduction}
Triangulated categories and exact categories are two fundamental structures in algebra, geometry and topology.
They are also important tools in many mathematical branches.
Nakaoka and Palu \cite{NP} introduced the notion
of extriangulated categories, whose extriangulated structures are given by $\E$-triangles with some axioms. Triangulated categories and exact categories are extriangulated categories. There are a lot of examples of extriangulated categories which are neither triangulated categories nor exact categories, see \cite{NP,ZZ,HZZ1,ZhZ,NP1}.

In \cite{GKO}, Geiss, Keller and Oppermann introduced $(n+2)$-angulated categories for any positive integer $n$. These
are a ``higher dimensional" analogue of triangulated categories,
 in the sense that triangles are replaced by $(n+2)$-angles, that is, morphism sequences of length $(n+2)$. Thus a $3$-angulated category is precisely a triangulated category.
An important source of examples of $(n+2)$-angulated categories are certain cluster tilting subcategories of triangulated categories.
Jasso \cite{Ja} introduced $n$-exact categories as
higher analogs of exact categories. Moreover, he also proved that any $n$-cluster-tilting subcategory of an
exact category is an $n$-exact category. Recently, Herschend, Liu and Nakaoka \cite{HLN} introduced the
notion of $n$-exangulated categories for any positive integer $n$. It is not only a higher
dimensional analogue of extriangulated categories defined by Nakaoka and Palu \cite{NP},
but also gives a common generalization of $n$-exact categories in the sense of
Jasso \cite{Ja} and $(n+2)$-angulated categories in the sense of Geiss-Keller-Oppermann \cite{GKO}.

Balmer and Schlichting \cite{BM} proved that the idempotent completion of a triangulated category admits a natural triangulated structure. B\"{u}hler \cite{B} showed that the idempotent completion of an exact category is still an exact category. Wang-Wei-Zhang-Zhao \cite{WWZZ} and Msapato \cite{M} independently unified their results of
Balmer and Schlichting \cite{BM} and  B\"{u}hler \cite{B} in the framework of extriangulated
categories. Specifically speaking, they proved that the idempotent completion of an extriangulated category is extriangulated. 
In order to construct more examples of $(n+2)$-angulated categories, Lin \cite{L2} extended Balmer and Schlichting's result from 3 to $(n+2)$. More precisely, Lin proved that the idempotent completion of an $(n+2)$-angulated category admits a natural $(n+2)$-angulated structure.
Let $(\C,\Sigma,\Theta)$ be a Krull-Schmidt $(n+2)$-angulated category and $\A$ be an $n$-extension closed subcategory of $\C$. Zhou \cite{Z} proved that $\A$ admits the structure of an $n$-exangulated category.

Based on their results of Lin \cite{L2} and Zhou \cite{Z}, we prove the following conclusion which gives $n$-exangulated categories which are neither $n$-exact nor $(n+2)$-angulated in general.

\begin{theorem} \rm (see Theorem \ref{main} for details)\label{mmm} Let $(\C,\Sigma,\Theta)$ be a Krull-Schmidt $(n+2)$-angulated category and
$\A$ be an $n$-extension closed subcategory of $\C$. Then the idempotent completion $\widetilde\A$ of $\A$ admits an $n$-exangulated structure.
\end{theorem}

Note that any $(n+2)$-angulated category $\C$ can be regarded as an $n$-extension closed subcategory of itself, hence Theorem \ref{mmm} generalizes the main result of Lin \cite{L2}. Our proof method is to avoid verifying that (EA1) holds, in fact, the verification of this axiom is complicated and difficult.

This article is organized as follows. In Section 2, we review some elementary definitions on $(n+2)$-angulated categories and $n$-exangulated categories. In Section 3, we prove our main result in this article.

\section{Preliminaries}
In this section, we briefly review basic concepts concerning $n$-exangulated categories
and $(n+2)$-angulated categories. Let $\C$ be an additive category and $n$ be a positive integer.
\subsection{$(n+2)$-angulated categories}
An $(n+2)$-$\Sigma$-sequence in $\C$ is a sequence of objects and morphisms
$$A_0\xrightarrow{f_0}A_1\xrightarrow{f_1}A_2\xrightarrow{f_2}\cdots\xrightarrow{f_{n-1}}A_n\xrightarrow{f_n}A_{n+1}\xrightarrow{f_{n+1}}\Sigma A_0.$$
Its {\em left rotation} is the $(n+2)$-$\Sigma$-sequence
$$A_1\xrightarrow{f_1}A_2\xrightarrow{f_2}A_3\xrightarrow{f_3}\cdots\xrightarrow{f_{n}}A_{n+1}\xrightarrow{f_{n+1}}\Sigma A_0\xrightarrow{(-1)^{n}\Sigma f_0}\Sigma A_1.$$
A \emph{morphism} of $(n+2)$-$\Sigma$-sequences is  a sequence of morphisms $\varphi=(\varphi_0,\varphi_1,\cdots,\varphi_{n+1})$ such that the following diagram commutes
$$\xymatrix{
A_0 \ar[r]^{f_0}\ar[d]^{\varphi_0} & A_1 \ar[r]^{f_1}\ar[d]^{\varphi_1} & A_2 \ar[r]^{f_2}\ar[d]^{\varphi_2} & \cdots \ar[r]^{f_{n}}& A_{n+1} \ar[r]^{f_{n+1}}\ar[d]^{\varphi_{n+1}} & \Sigma A_0 \ar[d]^{\Sigma \varphi_0}\\
B_0 \ar[r]^{g_0} & B_1 \ar[r]^{g_1} & B_2 \ar[r]^{g_2} & \cdots \ar[r]^{g_{n}}& B_{n+1} \ar[r]^{g_{n+1}}& \Sigma B_0
}$$
where each row is an $(n+2)$-$\Sigma$-sequence. It is an {\em isomorphism} if $\varphi_0, \varphi_1, \varphi_2, \cdots, \varphi_{n+1}$ are all isomorphisms in $\C$.

\begin{definition}\cite[Definition 2.1]{GKO}
An $(n+2)$-\emph{angulated category} is a triple $(\C, \Sigma, \Theta)$, where $\C$ is an additive category, $\Sigma$ is an auto-equivalence of $\C$ ($\Sigma$ is called the $n$-suspension functor), and $\Theta$ is a class of $(n+2)$-$\Sigma$-sequences (whose elements are called $(n+2)$-angles), which satisfies the following axioms:
\begin{itemize}[leftmargin=3em]
\item[\textbf{(N1)}]
\begin{itemize}
\item[(a)] The class $\Theta$ is closed under isomorphisms, direct sums and direct summands.

\item[(b)] For each object $A\in\C$ the trivial sequence
$$ A\xrightarrow{1_A}A\rightarrow 0\rightarrow0\rightarrow\cdots\rightarrow 0\rightarrow \Sigma A$$
belongs to $\Theta$.

\item[(c)] Each morphism $f_0\colon A_0\rightarrow A_1$ in $\C$ can be extended to $(n+2)$-$\Sigma$-sequence: $$A_0\xrightarrow{f_0}A_1\xrightarrow{f_1}A_2\xrightarrow{f_2}\cdots\xrightarrow{f_{n-1}}A_n\xrightarrow{f_n}A_{n+1}\xrightarrow{f_{n+1}}\Sigma A_0.$$
\end{itemize}
\item[\textbf{(N2)}] An $(n+2)$-$\Sigma$-sequence belongs to $\Theta$ if and only if its left rotation belongs to $\Theta$.

\item[\textbf{(N3)}] Each solid commutative diagram
$$\xymatrix{
A_0 \ar[r]^{f_0}\ar[d]^{\varphi_0} & A_1 \ar[r]^{f_1}\ar[d]^{\varphi_1} & A_2 \ar[r]^{f_2}\ar@{-->}[d]^{\varphi_2} & \cdots \ar[r]^{f_{n}}& A_{n+1} \ar[r]^{f_{n+1}}\ar@{-->}[d]^{\varphi_{n+1}} & \Sigma A_0 \ar[d]^{\Sigma\varphi_0}\\
B_0 \ar[r]^{g_0} & B_1 \ar[r]^{g_1} & B_2 \ar[r]^{g_2} & \cdots \ar[r]^{g_{n}}& B_{n+1} \ar[r]^{g_{n+1}}& \Sigma B_0
}$$ with rows in $\Theta$, the dotted morphisms exist and give a morphism of  $(n+2)$-$\Sigma$-sequences.

\item[\textbf{(N4)}] In the situation of (N3), the morphisms $\varphi_2,\varphi_3,\cdots,\varphi_{n+1}$ can be chosen such that the mapping cone
$$A_1\oplus B_0\xrightarrow{\left(\begin{smallmatrix}
                                        -f_1&0\\
                                        \varphi_1&g_0
                                       \end{smallmatrix}
                                     \right)}
A_2\oplus B_1\xrightarrow{\left(\begin{smallmatrix}
                                        -f_2&0\\
                                        \varphi_2&g_1
                                       \end{smallmatrix}
                                     \right)}\cdots\xrightarrow{\left(\begin{smallmatrix}
                                        -f_{n+1}&0\\
                                        \varphi_{n+1}&g_n
                                       \end{smallmatrix}
                                     \right)} \Sigma A_0\oplus B_{n+1}\xrightarrow{\left(\begin{smallmatrix}
                                        -\Sigma f_0&0\\
                                        \Sigma\varphi_1&g_{n+1}
                                       \end{smallmatrix}
                                     \right)}\Sigma A_1\oplus\Sigma B_0$$
belongs to $\Theta$.
   \end{itemize}
\end{definition}

\begin{remark}\label{rem1}
(a) From \cite{GKO}, we know that the classical triangulated categories are the special case $n=1$.

(b)  The composition of two consecutive morphisms in an $(n+2)$-angle is zero, see \cite[Lemma 3.1]{BT}.
\end{remark}

\begin{lemma}\emph{\cite[Lemma 3.13]{F}}\label{y1}
Let $(\C, \Sigma, \Theta)$ be an $(n+2)$-angulated category, and
\begin{equation}\label{t1}
\begin{array}{l}
A_0\xrightarrow{f_0}A_1\xrightarrow{f_1}A_2\xrightarrow{f_2}\cdots\xrightarrow{f_{n-1}}A_n\xrightarrow{f_n}A_{n+1}\xrightarrow{f_{n+1}}\Sigma A_0.\end{array}
\end{equation}
be an $(n+2)$-angle in $\C$. Then the following statments are equivalent:
\begin{itemize}
\item[\rm (1)] $f_0$ is a section (also known as a split monomorphism);

\item[\rm (2)] $f_n$ is a retraction (also known as a split epimorphism);

\item[\rm (3)] $f_{n+1}=0$.
\end{itemize}
If an $(n+2)$-angle \emph{(\ref{t1})} satisfies one of the above equivalent conditions, it is called \emph{split}.
\end{lemma}

The following lemma exhibits a crucial difference between $(n+2)$-angulated and
$n$-exact categories.

\begin{lemma}\label{lemma}
Let $(\C, \Sigma, \Theta)$ be an $(n+2)$-angulated category.
Then any monomorphism in $\C$ is a section.
Dually any epimorphism $\C$ is a retraction.
\end{lemma}

\proof Assume that $f_0$ is a monomorphism. By (N1), there exists an $(n+2)$-angle
$$A_0\xrightarrow{f_0}A_1\xrightarrow{f_1}A_2\xrightarrow{f_2}\cdots\xrightarrow{f_{n-1}}A_n\xrightarrow{f_n}A_{n+1}\xrightarrow{f_{n+1}}\Sigma A_0$$
in $\C$. By Remark \ref{rem1}, we have $f_0\circ\Sigma^{-1} f_{n+1}=0$.
Since $f_0$ is a monomorphism, we obtain $\Sigma^{-1} f_{n+1}=0$ and then $f_{n+1}=0$.
By Lemma \ref{y1}, we have that $f_0$ is a section.

Dually we can show that any epimorphism $\C$ is a retraction.  \qed

\subsection{$n$-exangulated categories}
 Suppose that $\C$ is equipped with an additive bifunctor $\E\colon\C\op\times\C\to{\rm Ab}$, where ${\rm Ab}$ is the category of abelian groups. Next we briefly recall some definitions and basic properties of $n$-exangulated categories from \cite{HLN}. We omit some
details here, but the reader can find them in \cite{HLN}.

{ For any pair of objects $A,C\in\C$, an element $\del\in\E(C,A)$ is called an {\it $\E$-extension} or simply an {\it extension}. We also write such $\del$ as ${}_A\del_C$ when we indicate $A$ and $C$. The zero element ${}_A0_C=0\in\E(C,A)$ is called the {\it split $\E$-extension}. For any pair of $\E$-extensions ${}_A\del_C$ and ${}_{A'}\del{'}_{C'}$, let $\delta\oplus \delta'\in\mathbb{E}(C\oplus C', A\oplus A')$ be the
element corresponding to $(\delta,0,0,{\delta}{'})$ through the natural isomorphism $\mathbb{E}(C\oplus C', A\oplus A')\simeq\mathbb{E}(C, A)\oplus\mathbb{E}(C, A')
\oplus\mathbb{E}(C', A)\oplus\mathbb{E}(C', A')$.

For any $a\in\C(A,A')$ and $c\in\C(C',C)$,  $\E(C,a)(\del)\in\E(C,A')\ \ \text{and}\ \ \E(c,A)(\del)\in\E(C',A)$ are simply denoted by $a_{\ast}\del$ and $c^{\ast}\del$, respectively.

Let ${}_A\del_C$ and ${}_{A'}\del{'}_{C'}$ be any pair of $\E$-extensions. A {\it morphism} $(a,c)\colon\del\to{\delta}{'}$ of extensions is a pair of morphisms $a\in\C(A,A')$ and $c\in\C(C,C')$ in $\C$, satisfying the equality
$a_{\ast}\del=c^{\ast}{\delta}{'}$.}
Then the functoriality of $\E$ implies $\E(c,a)=a_{\ast}(c^{\ast}\del)=c^{\ast}(a_{\ast}\del)$.

\begin{definition}\cite[Definition 2.7]{HLN}
Let $\bf{C}_{\C}$ be the category of complexes in $\C$. As its full subcategory, define $\CC$ to be the category of complexes in $\C$ whose components are zero in the degrees outside of $\{0,1,\ldots,n+1\}$. Namely, an object in $\CC$ is a complex $X^{\mr}=\{X_i,d^X_i\}$ of the form
\[ X_0\xrightarrow{d^X_0}X_1\xrightarrow{d^X_1}\cdots\xrightarrow{d^X_{n-1}}X_n\xrightarrow{d^X_n}X_{n+1}. \]
We write a morphism $f^{\mr}\co X^{\mr}\to Y^{\mr}$ simply $f^{\mr}=(f^0,f^1,\ldots,f^{n+1})$, only indicating the terms of degrees $0,\ldots,n+1$.
\end{definition}

\begin{definition}\cite[Definition 2.23]{HLN}
Let $\s$ be an exact realization of $\E$.
\begin{enumerate}
\item[\rm (1)] An $n$-exangle $\Xd$ is called an $\s$-{\it distinguished} $n$-exangle if it satisfies $\s(\del)=[X^{\mr}]$. We often simply say {\it distinguished $n$-exangle} when $\s$ is clear from the context.
\item[\rm (2)]  An object $X^{\mr}\in\CC$ is called an {\it $\s$-conflation} or simply a {\it conflation} if it realizes some extension $\del\in\E(X_{n+1},X_0)$.
\item[\rm (3)]  A morphism $f$ in $\C$ is called an {\it $\s$-inflation} or simply an {\it inflation} if it admits some conflation $X^{\mr}\in\CC$ satisfying $d_X^0=f$.
\item[\rm (4)]  A morphism $g$ in $\C$ is called an {\it $\s$-deflation} or simply a {\it deflation} if it admits some conflation $X^{\mr}\in\CC$ satisfying $d_X^n=g$.
\end{enumerate}
\end{definition}

\begin{definition}\cite[Definition 2.32]{HLN}
An {\it $n$-exangulated category} is a triplet $(\C,\E,\s)$ of additive category $\C$, additive bifunctor $\E\co\C\op\times\C\to\Ab$, and its exact realization $\s$, satisfying the following conditions.
\begin{itemize}[leftmargin=3.3em]
\item[{\rm (EA1)}] Let $A\ov{f}{\lra}B\ov{g}{\lra}C$ be any sequence of morphisms in $\C$. If both $f$ and $g$ are inflations, then so is $g\circ f$. Dually, if $f$ and $g$ are deflations, then so is $g\circ f$.

\item[{\rm (EA2)}] For $\rho\in\E(D,A)$ and $c\in\C(C,D)$, let ${}_A\langle X^{\mr},c\uas\rho\rangle_C$ and ${}_A\Yr_D$ be distinguished $n$-exangles. Then $(\id_A,c)$ has a {\it good lift} $f^{\mr}$, in the sense that its mapping cone gives a distinguished $n$-exangle $\langle M^{\mr}_f,(d^X_0)\sas\rho\rangle$.
\end{itemize}
\begin{itemize}[leftmargin=4.3em]
\item[{\rm (EA2$\op$)}] Dual of {\rm (EA2)}.
\end{itemize}
Note that the case $n=1$, a triplet $(\C,\E,\s)$ is a  $1$-exangulated category if and only if it is an extriangulated category, see \cite[Proposition 4.3]{HLN}.
\end{definition}

\begin{example}
From \cite[Proposition 4.34]{HLN} and \cite[Proposition 4.5]{HLN},  we know that $n$-exact categories and $(n+2)$-angulated categories are $n$-exangulated categories.
There are some other examples of $n$-exangulated categories
 which are neither $n$-exact nor $(n+2)$-angulated, see \cite{HLN,HLN1,LZ,HZZ}.
\end{example}

\section{Main result}
In this section, when we say that $\A$ is a subcategory of an additive category $\C$, we always assume that $\A$ is full, and closed under isomorphisms, direct sums and direct summands.

An additive category $\C$ is said to be  {\em idempotent complete} if for each object $A$ in $\C$ and for each idempotent $e:A\rightarrow A$, we have $A=\Im(e)\oplus \Ker(e)$.

\begin{definition}\cite[Definition 1.2]{BM}
 Let $\C$ be an additive category. The {\em idempotent completion} of $\C$ is the category $\widetilde{\C}$ defined as follows:

  Objects of $\widetilde{\C}$ are pairs $\widetilde{A}=(A,e_{a})$, where $A$ is an object of $\C$ and $e_{a}: A\rightarrow A$ is an idempotent morphism.

  A morphism in $\widetilde{\C}$ from $(A,e_{a})$ to $(B,e_{b})$ is a morphism $p: A\rightarrow B$ in $\C$ such that $p e_{a}=e_{b}p=p$. That is to say, we have the following commutative diagram
\[
 \begin{tikzpicture}
 \draw (0,0) node{\xymatrix{
 A \ar[r]^{{e_{a}}} \ar[dr]|{p}\ar@{}[dr] \ar[d]_{p} & A \ar[d]^{p}\\
 B  \ar[r]_{e_{b}} &B.
}};
\draw (0.2,0.2) node{\tiny $\circlearrowright$};
\draw (-0.2,-0.2) node{\tiny $\circlearrowleft$};
 \end{tikzpicture}
\]
\end{definition}

\begin{remark}\label{re0}
\begin{enumerate}
\item[\rm (1)] The assignment $A\mapsto (A,1)$ defines a fully faithful additive functor $\iota: \C\rightarrow\widetilde{\C}$. Namely, we can view $\C$ as a full subcategory of $\widetilde{\C}$.
\item[\rm (2)] For each object $X\in\widetilde{\C }$, there exists an object $X'\in\widetilde{\C}$ such that $X\oplus X'\in\C$. In fact, if $X=(A,e)$, then we can take $X'=(A,1-e)$ and $X\oplus X'\cong A\in\C$.

\item[\rm (3)] From the definition,  it is obvious that if $\C$ is Krull-Schmidt,
then $\widetilde{\C}$ is also Krull-Schmidt.

\end{enumerate}
\end{remark}
\begin{definition}\cite[Definition 3.6]{L1}
Let $(\C,\Sigma,\Theta)$ be an $(n+2)$-angulated category. A subcategory $\A$ of $\C$ is called \emph{$n$-extension closed} if
 for each morphism $f_{n+1}\colon A_{n+1}\to \Sigma A_0$ with $A_0,A_{n+1}\in\A$,
there exists an $(n+2)$-angle
$$A_0\xrightarrow{f_0}A_1\xrightarrow{f_1}A_2\xrightarrow{f_2}\cdots\xrightarrow{f_{n-1}}A_n\xrightarrow{f_n}A_{n+1}\xrightarrow{f_{n+1}}\Sigma A_0$$
with terms $A_1,A_2,\cdots,A_n\in\A$.
\end{definition}

\begin{remark}\label{remark2}
(a) From the definition, we know that $\A$ is not necessarily closed under $\Sigma$.
That is, if $A\in\A$, we don't require  $\Sigma A\in\A$ or $\Sigma^{-1}A\in\A$.

(b) Any $(n+2)$-angulated category can be seen as an $n$-extension closed subcategory of itself.
\end{remark}

The following two key lemmas are essentially contained in \cite{L2}. We provide a proof for the reader's convenience.

Let $\C$ be an additive category with an automorphism $\Sigma\colon \C\to \C$ and $\Theta$ be a class of $(n+2)$-$\Sigma$-sequences.

\begin{lemma}\label{y0}
Suppose that $\Theta$ satisfies {\rm (N1)(b), (N2), (N3)} and the $(n+2)$-$\Sigma$-sequence
$$A_0\xrightarrow{f_0}A_1\xrightarrow{f_1}A_2\xrightarrow{f_2}\cdots\xrightarrow{f_{n-1}}A_n\xrightarrow{f_n}A_{n+1}\xrightarrow{f_{n+1}}\Sigma A_0$$
in $\Theta$. Then the following statements hold.

{\rm (1)} $f_{n+1}f_{n}=0$.

{\rm (2)} If $f_{n+1}g_{n+1}=0$ for some morphism $g_{n+1}:B_{n+1}\rightarrow A_{n+1}$, then there exists a morphism $h:B_{n+1}\rightarrow A_{n}$ such that $g_{n+1}=f_{n}h$.
\end{lemma}
\begin{proof}
(1) By (N1)(b), (N2) and (N3), we have the following commutative diagram of $(n+2)$-$\Sigma$-sequence
$$\xymatrix{
A_0 \ar[r]^{f_0}\ar@{=}[d]^{} & A_1 \ar[r]^{f_1}\ar[d]^{} & A_2 \ar[r]^{f_2}\ar@{-->}[d]^{} & \cdots \ar[r]^{f_{n-1}}&A_{n}\ar@{-->}[d]^{}\ar[r]^{f_{n}} & A_{n+1} \ar[r]^{f_{n+1}}\ar@{-->}[d]^{f_{n+1}} & \Sigma A_0 \ar@{=}[d]^{}\\
A_0 \ar[r]^{} &0 \ar[r]^{} & 0 \ar[r]^{} & \cdots \ar[r]^{}& 0\ar[r]^{}& \Sigma A_0 \ar[r]^{\id}& \Sigma A_0.
}$$ So we have $f_{n+1}f_{n}=0$.

(2) Since $f_{n+1}g_{n+1}=0$, by (N1)(b), (N2) and (N3), we have the following commutative diagram of $(n+2)$-$\Sigma$-sequence
$$\xymatrix{
0 \ar[r]^{}\ar[d]^{} & 0 \ar[r]^{}\ar@{-->}[d]^{} & 0 \ar[r]^{}\ar@{-->}[d]^{} & \cdots \ar[r]^{}& B_{n+1}\ar@{-->}[d]^{h}\ar[r]^{\id} & B_{n+1} \ar[r]^{}\ar[d]^{g_{n+1}} & 0 \ar[d]^{}\\
A_0 \ar[r]^{f_0} & A_1 \ar[r]^{f_1}& A_2 \ar[r]^{f_2} & \cdots \ar[r]^{f_{n-1}}&A_{n}\ar[r]^{f_{n}} & A_{n+1} \ar[r]^{f_{n+1}} & \Sigma A_0 .
}$$ Hence there exists a morphism $h:B_{n+1}\rightarrow A_{n}$ such that $g_{n+1}=f_{n}h$.
\end{proof}

\begin{lemma}\label{y1}
Suppose that $\C$ is idempotent complete, $\Theta$ satisfies {\rm (N1)(b), (N2), (N3)} and
$$A_{\bullet}=(A_0\xrightarrow{f_0}A_1\xrightarrow{f_1}\cdots\xrightarrow{f_{n-2}}A_{n-1}\xrightarrow{f_{n-1}}A_n\xrightarrow{\binom{f_n}{g_n}}A_{n+1}\oplus B_{n+1}\xrightarrow{(f_{n+1},\hspace{0.8mm} g_{n+1})}\Sigma A_0)\in\Theta.$$

If $g_{n+1}=0$, then $A_{\bullet}\simeq A'_{\bullet}\oplus B_{\bullet}$, where
$$A'_{\bullet}=(A_0\xrightarrow{f_0}A_1\xrightarrow{f_{1}}\cdots\xrightarrow{f_{n-2}}A_{n-1}\xrightarrow{m}A'_n\xrightarrow{k_1}A_{n+1}\xrightarrow{f_{n+1}}\Sigma A_0)$$
and
$$B_{\bullet}=(0\xrightarrow{}0\xrightarrow{}0\xrightarrow{}\cdots\xrightarrow{}B_{n+1}\xrightarrow{\id}B_{n+1}\xrightarrow{}0).$$

\end{lemma}
\begin{proof}
Since $(f_{n+1}, g_{n+1})\binom{0}{1}=g_{n+1}=0,$ by Lemma \ref{y0} (2), there exists a morphism $g'_{n+1}:B_{n+1}\rightarrow A_{n}$, such that $\binom{0}{1}=\binom{f_{n}}{g_{n}}g'_{n+1}$, so we have $g_{n}g'_{n+1}=1$. Since $(g'_{n+1}g_{n})^{2}=g'_{n+1}g_{n}$ and $C$ is idempotent complete, we can write $A_{n}=A'_{n}\oplus A''_{n}$ and
$$A_{\bullet}=(A_0\xrightarrow{f_0}A_1\xrightarrow{f_1}\cdots\xrightarrow{f_{n-2}}A_{n-1}\xrightarrow{\binom{m}{l}}A'_{n}\oplus A''_{n}\xrightarrow{\left(
                                                           \begin{smallmatrix}
                                                             k_{1} & k_{2} \\
                                                             0 & k_{3} \\
                                                           \end{smallmatrix}
                                                         \right)}A_{n+1}\oplus B_{n+1}\xrightarrow{(f_{n+1}, 0)}\Sigma A_0)\in\Theta,$$
where $k_{3}$ is an isomorphism. By Lemma \ref{y0} (1) and (N2), we have $l=0$ and $f_{n+1}k_2=0$. Since $(f_{n+1}, 0)\binom{k_2}{0}=0,$ there exists a morphism $\binom{a}{b}:A''_{n}\rightarrow A'_{n}\oplus A''_{n}$ such that $\binom{k_2}{0}=\left(
                                                           \begin{smallmatrix}
                                                             k_{1} & k_{2} \\
                                                             0 & k_{3} \\
                                                           \end{smallmatrix}
                                                         \right)\binom{a}{b}$ by Lemma \ref{y0} (2).
Thus $b=0$ and $k_1a=k_2$, so we have the following commutative diagram
$$\xymatrix{
A_0 \ar[r]^{f_0}\ar@{=}[d]^{} & A_1 \ar[r]^{f_1}\ar@{=}[d]^{} & \cdots \ar[r]^{f_{n-2}}&A_{n-1}\ar@{=}[d]^{}\ar[r]^{\binom{m}{0}} &A'_{n}\oplus A''_{n}\ar[d]^{\left(
                                                           \begin{smallmatrix}
                                                             1 & a \\
                                                             0 & k_{3} \\
                                                           \end{smallmatrix}
                                                         \right)}\ar[r]^{\left(
                                                           \begin{smallmatrix}
                                                             k_{1} & k_{2} \\
                                                             0 & k_{3} \\
                                                           \end{smallmatrix}
                                                         \right)} & A_{n+1}\oplus B_{n+1} \ar[r]^{(f_{n+1},0)}\ar@{=}[d]^{} & \Sigma A_0 \ar@{=}[d]^{}\\
A_0 \ar[r]^{f_0} &A_1  \ar[r]^{f_1} & \cdots \ar[r]^{f_{n-2}}&A_{n-1}\ar[r]^{\binom{m}{0}} & A'_{n}\oplus B_{n+1}\ar[r]^{\left(
                                                           \begin{smallmatrix}
                                                             k_{1} & 0 \\
                                                             0 & 1 \\
                                                           \end{smallmatrix}
                                                         \right)}& A_{n+1}\oplus B_{n+1} \ar[r]^{(f_{n+1},0)}& \Sigma A_0
}$$
which shows that $A_{\bullet}\simeq A'_{\bullet}\oplus B_{\bullet}$.

\end{proof}

Let $(\C,\Sigma,\Theta)$ be an $(n+2)$-angulated category. Since $\Sigma\colon\C\xrightarrow{~\simeq~}\C$ is an  automorphism, then
$\Sigma$ gives an additive bifunctor
$$\E_{\Sigma}=\C(-,\Sigma-)\colon \C^{\rm op}\times \C\to {\rm Ab},$$
defined by the following.
\begin{itemize}
\item[\rm (i)] For any $A,C\in\C$, $\E_{\Sigma}(C, A)=\C(C,\Sigma A)$;

\item[\rm (ii)] For any $a\in\C(A,A')$ and $c\in\C(C',C)$, the map $\E_{\Sigma}(c, a)\colon\C(C, \Sigma A)\to \C(C', \Sigma A')$
sends $\delta\in\C(C, \Sigma A)$ to $c^{\ast}a_{\ast}\delta=(\Sigma a)\circ\delta\circ c$.
\end{itemize}

For each $\delta\in\E_{\Sigma}(C, A)$, we complete
it into an $(n+2)$-angle
$$A\xrightarrow{f_0}X_1\xrightarrow{f_1}X_2\xrightarrow{f_2}\cdots\xrightarrow{f_{n-1}}X_n\xrightarrow{f_n}C\xrightarrow{\delta}\Sigma A_0$$
Define $\s_{\Theta}(\delta)=[X^{\mr}]$ by using $X^{\mr}\in{\mathbf{C}^{n+2}_{(A,\hspace{0.8mm}C)}}$ given by
$$A\xrightarrow{f_0}X_1\xrightarrow{f_1}X_2\xrightarrow{f_2}\cdots\xrightarrow{f_{n-1}}X_n\xrightarrow{f_n}C$$

The following result shows that any $(n+2)$-angulated category can be viewed as
an $n$-exangulated category.
\begin{theorem}{\rm \cite[Proposition 4.5]{HLN}}
With the above definition, $(\C,\E_{\Sigma},\s_{\Theta})$ is an $n$-exangulated category.
\end{theorem}

Let $(\C,\Sigma,\Theta)$ be an $(n+2)$-angulated category and
$\A$ be an $n$-extension closed subcategory of $\C$.
Define $\E_{\A}$ to be the restriction of
$\E_{\Sigma}$ onto $\A^{\rm op}\times\A$, and define $\s_{\A}$ by restricting $\s_{\Theta}$.
Recall that an additive category $\C$ is \emph{Krull-Schmidt} if any object $A$ has a decomposition  $A=A_1\oplus A_2\oplus\cdots\oplus A_n$ such that each $A_i$ is indecomposable with local endomorphism ring.

\begin{theorem}\rm\label{main1}\cite[Theorem 3.4]{Z}
Let $(\C,\Sigma,\Theta)$ be a Krull-Schmidt $(n+2)$-angulated category and
$\A$ be an $n$-extension closed subcategory of $\C$. Then $(\A,\E_{\A},\s_{\A})$
is an $n$-exangulated category.
\end{theorem}

\begin{remark}
In Theorem \ref{main1}, $(\A,\E_{\A},\s_{\A})$
is an $n$-exangulated category which is neither $n$-exact nor $(n+2)$-angulated \textbf{in general}.
\begin{itemize}
\item[\rm (a)] If $(\A,\E_{\A},\s_{\A})$ is $n$-exact, by \cite[Proposition 4.37]{HLN1}, we know that
any inflation is monomorphism. By Lemma \ref{lemma},
any monomorphism is section. This shows that all distinguished $n$-exangles are split.
This is not possible in general.

\item[(b)]  Since $\A$ is not necessarily closed under $\Sigma$, by \cite[Proposition 4.8]{HLN1},
we know that $(\A,\E_{\A},\s_{\A})$ is not $(n+2)$-angulated in general.

\end{itemize}
\end{remark}

\begin{theorem}\rm\label{main2}\cite[Theorem 3.1]{L2}
Let $(\C,\Sigma,\Theta)$ be an $(n+2)$-angulated category and
$\widetilde{\C}$ the idempotent completion of $\C$. Denote by $\widetilde{\Sigma}$ the isomorphism of $\widetilde{\C}$ induced by the isomorphism $\Sigma$ of $\C$ and by $\widetilde{\Theta}$ the class of $(n+2)$-$\widetilde{\Sigma}$-sequences in $\widetilde{\C}$ that are direct summands of $(n+2)$-${\Sigma}$-sequences in $\Theta$. Then $(\widetilde{\C},\widetilde{\Sigma},\widetilde{\Theta})$
is an  $(n+2)$-angulated category.
\end{theorem}
\begin{remark}\rm\label{rem}
In Theorem \ref{main2}, $\widetilde{\Sigma}$ is induced by the isomorphism $\Sigma$ of $\C$, more precisely, for any object $\widetilde{C}=(C,e_c)$ in $\widetilde{\C}$, Define $\widetilde{\Sigma}:\widetilde{\C}\rightarrow\widetilde{\C}$ by $\widetilde{\Sigma}(\widetilde{C})=\widetilde{\Sigma}(C,e_c)=(\Sigma C,\Sigma (e_c))$.

\end{remark}

Our main result is the following.
\begin{theorem}\rm\label{main} Let $(\C,\Sigma,\Theta)$ be a Krull-Schmidt $(n+2)$-angulated category and
$\A$ be an $n$-extension closed subcategory of $\C$. Then the idempotent completion $\widetilde\A$ of $\A$ admits an $n$-exangulated structure.
\end{theorem}

\begin{proof} By Theorem \ref{main2}, we know that the idempotent completion $(\widetilde{\C},\widetilde{\Sigma},\widetilde{\Theta})$ of $(\C,\Sigma,\Theta)$ is an  $(n+2)$-angulated category. We claim that $\widetilde\A$ is an $n$-extension closed subcategory of $\widetilde{\C}$. Indeed, for each morphism $f_{n+1}:{A_{n+1}}\rightarrow \widetilde{\Sigma}{A_{0}}$ with ${A_{n+1}},{A_{0}}\in\widetilde{\A}$, we can choose two objects
${A'_{n+1}},{A'_{0}}\in\widetilde{\A}$ such that ${A_{n+1}}\oplus{A'_{n+1}},{A_{0}}\oplus{A'_{0}}\in{\A}$ by Remark \ref{re0}. Note that $\left(
\begin{smallmatrix}                                                                                                                                                    f_{n+1} & 0 \\                                                                                                                                                    0 & 0 \\                                                                                                                                                  \end{smallmatrix}                                                                                                                                                \right):{A_{n+1}}\oplus
{A'_{n+1}}\rightarrow {\Sigma}{A_{0}}\oplus{\Sigma}{A'_{0}}$ is a morphism in $\C$. Since $\A$ is an $n$-extension closed subcategory of $\C$,  there exists an $(n+2)$-angle
$$\widetilde{X}_\bullet=({A_{0}}\oplus {A'_{0}}\xrightarrow{f_{0}
  }{A_{1}}\xrightarrow{f_1} {A_{2}}\xrightarrow{f_3}\cdots\xrightarrow{f_{n-1}}A_{n}\xrightarrow{{f_{n}}}{{A_{n+1}}\oplus
{A'_{n+1}}}\xrightarrow{\left(                                                                                                                                                  \begin{smallmatrix}                                                                                                                                                    f_{n+1} & 0 \\                                                                                                                                                    0 & 0 \\                                                                                                                                                  \end{smallmatrix}                                                                                                                                                \right)} {\Sigma} {A_{0}}\oplus {\Sigma} {A'_{0}})$$
in $\Theta$, where $A_1,A_2,\cdots,A_n\in\A$. Since $\widetilde\C$ is idempotent complete and $\widetilde{X}_\bullet\in\widetilde\Theta$, It follows that $\widetilde{X}_\bullet\cong X_\bullet\oplus X'_\bullet$ by Lemma \ref{y1}, where
 $$X_\bullet=(A_0\xrightarrow{f'_0}A'_1\xrightarrow{f_1'} A_2\xrightarrow{f_3} \cdots\xrightarrow{f_{n-2}} A_{n-1}\xrightarrow{f_{n-1}'} A_n'\xrightarrow{f_n'}A_{n+1}\xrightarrow{f_{n+1}}\widetilde{\Sigma} A_0),$$
 $$X'_\bullet=(A_0'\xrightarrow{\id} A_0'\xrightarrow{0} 0\rightarrow 0\rightarrow\cdots\rightarrow 0\xrightarrow{0} A'_{n+1}\xrightarrow{\id}A'_{n+1}\xrightarrow{0}\widetilde{\Sigma} A_0').$$
So $X_\bullet$ belongs to $\widetilde{\Theta}$. Note that $A'_1,A_2,\cdots,A'_n\in\A$.  By Remark \ref{re0}, we have that $\widetilde\A$ is an $n$-extension closed subcategory of $\widetilde{\C}$. Therefore, $\widetilde\A$ admits an $n$-exangulated structure by Theorem \ref{main1}.
\end{proof}

\begin{remark} The proof of Theorem \ref{main},
we avoid verifying that (EA1) holds, in fact, the verification of this axiom is complicated and difficult.
\end{remark}

\textbf{Jian He}\\
Department of Applied Mathematics, Lanzhou University of Technology, Lanzhou 730050, Gansu, P. R. China\\
E-mail: \textsf{jianhe30@163.com}\\[0.3cm]
\textbf{Jing He}\\
College of Science, Hunan University of Technology and Business, 410205 Changsha, Hunan P. R. China\\
E-mail: \textsf{jinghe1003@163.com}\\[0.3cm]
\textbf{Panyue Zhou}\\
College of Mathematics, Hunan Institute of Science and Technology, Yueyang 414006, Hunan, P. R. China\\
E-mail: \textsf{panyuezhou@163.com}

\end{document}